\documentclass{amsart}

\newtheorem{theorem}{Theorem}[section]
\newtheorem{lemma}[theorem]{Lemma}

\newtheorem{corollary}[theorem]{Corollary}
\newtheorem{proposition}[theorem]{Proposition}
\usepackage{graphicx,amssymb}

\theoremstyle{definition}

\newtheorem{example}[theorem]{Example}

\theoremstyle{remark}
\newtheorem{remark}[theorem]{Remark}

\numberwithin{equation}{section}

\begin{document}

\title[Left-orderability]{Left-orderable fundamental group and Dehn surgery on genus one two-bridge knots}

\author{Ryoto Hakamata}
\address{Graduate School of Education, Hiroshima University,
1-1-1 Kagamiyama, Higashi-hiroshima, Japan 739-8524.}

\author{Masakazu Teragaito}
\address{Department of Mathematics and Mathematics Education, Hiroshima University,
1-1-1 Kagamiyama, Higashi-hiroshima, Japan 739-8524.}
\email{teragai@hiroshima-u.ac.jp}
\thanks{The second author is partially supported by Japan Society for the Promotion of Science,
Grant-in-Aid for Scientific Research (C), 22540088.
}%





\begin{abstract}
For any hyperbolic genus one $2$-bridge knot in the $3$-sphere,
we show that the resulting manifold by $r$-surgery on the knot has left-orderable
fundamental group if the slope $r$ lies in some range which depends on the knot.
\end{abstract}

\maketitle

\section{Introduction}

A non-trivial group $G$ is said to be \textit{left-orderable\/} if it admits
a strict total ordering which is invariant under left-multiplication.
Thus, if $g<h$ then $fg<fh$ for any $f,g,h\in G$.
Many fundamental groups of $3$-manifolds are known to be left-orderable.
For examples, all knot and link groups are left-orderable.
Boyer, Gordon and Watson \cite{BGW} propose a conjecture that
an irreducible rational homology $3$-sphere $Y$ is an $L$-space if and only if $\pi_1Y$ is not left-orderable. 
An $L$-space is a rational homology $3$-sphere whose Heegaard--Floer homology $\widehat{HF}(Y)$
has rank $|H_1(Y;\mathbb{Z})|$ (\cite{OS}).
Recently, $L$-spaces become the object of interest, and it is an open problem to characterize
$L$-spaces without mentioning Heegaard--Floer homology.
The affirmative answer to the above conjecture will give
an algebraic characterization of $L$-spaces.

On the other hand, 
irreducible rational homology spheres are obtained by Dehn surgery on knots
in the $3$-sphere $S^3$, in plenty.
For a knot $K$, we call
a slope $r$ \textit{left-orderable\/} if the resulting manifold $K(r)$ by $r$-surgery on $K$ has left-orderable fundamental group. 
In this paper, a slope is sometimes identified with its parameter in $\mathbb{Q}\cup\{1/0\}$.
In particular, the meridional slope corresponds to $1/0$.
It is known that any hyperbolic $2$-bridge knot does not admit
Dehn surgery yielding an $L$-space (\cite{OS}).
Hence any slope but $1/0$ is expected to be left-orderable for a hyperbolic $2$-bridge knot, if we support the above conjecture.
In this direction, Boyer, Gordon and Watson \cite{BGW} proved that for the figure-eight knot,
if $r\in (-4,4)$, then $r$ is left-orderable.
Later, Clay, Lidman and Watson \cite{CLW} showed $r=\pm 4$ are also left-orderable.
These results were extended to all hyperbolic twist knots \cite{HT0,HT1,T}.
Also, Tran \cite{Tr} further extended the range of left-orderable slopes for twist knots.

The purpose of this paper is to give
new ranges of
left-orderable slopes for all hyperbolic genus one, $2$-bridge knots.
As a by-product, we obtain a wider range of left-orderable slopes for (positive) twist knots than
that in \cite{HT1}.

For non-zero integers $m$ and $n$,
let $K(m,n)$ be the $2$-bridge knot $S(4mn+1,2m)$ in Schubert's normal form
as illustrated in Figure \ref{fig:knot}.
In Figure \ref{fig:knot},
the twists in the vertical box are left-handed (resp. right-handed) if $m>0$ (resp. $m<0$),
but those in the horizontal box are 
right-handed (resp. left-handed) if $n>0$ (resp. $n<0$).
Thus $K(1,-1)$ is the trefoil, and $K(1,1)$ is the figure-eight knot.
By symmetry, $K(m,n)$ and $K(-n,-m)$ are isotopic.
Except the trefoil, $K(m,n)$ is hyperbolic.
It is also well known that any genus one $2$-bridge knot is equivalent to $K(m,n)$
for some $m,n$ (see \cite{BZ}).

\begin{figure}[ht]
\includegraphics*[scale=0.4]{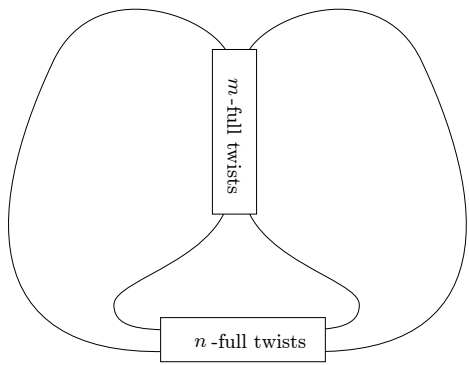}
\caption{The knot $K(m,n)$}\label{fig:knot}
\end{figure}

\begin{theorem}\label{thm:main}
Let $K=K(m,n)$ be a hyperbolic genus one $2$-bridge knot $S(4mn+1,2m)$ in the $3$-sphere $S^3$ as illustrated in Figure \ref{fig:knot}. 
Let $I$ be the interval defined by
\[I=
\begin{cases}
(-4n,4m) &\text{if $m>0$ and $n>0$},\\
(4m,-4n) & \text{if $m<0$ and $n<0$},\\
[0,\max\{4m,-4n\}) & \text{if $m>0$ and $n<0$},\\
(\min\{4m,-4n\},0] & \text{if $m<0$ and $n>0$}.
\end{cases}
\]
Then any slope in $I$ is left-orderable.
That is, $\pi_1(K(r))$ is left-orderable.
\end{theorem}

We remark that $K(1,-1)$ and $K(-1,1)$
are trefoils, so they are excluded in the statement here.

In a previous paper  \cite{HT1}, we showed that
any hyperbolic twist knot $K(1,n)\ (n\ne -1)$ admits a range $[0,4]$ of left-orderable slopes.
Theorem \ref{thm:main} gives a partial improvement of this result.

\begin{corollary}\label{cor}
Let $K=K(1,n)$ be a hyperbolic $n$-twist knot for $n>0$. 
If $I=(-4n,4]$, then
any slope in $I$ is left-orderable.
\end{corollary}

This range of Corollary \ref{cor}
coincides with that in \cite{Tr}.
Also, Tran informed us that he obtained a similar result to Theorem \ref{thm:main}.
We would like to thank Anh T. Tran for informing his result.

\section{Knot groups and Two sequences of polynomials}\label{sec:knotgroup}

Let $K=K(m,n)$ and let $G=\pi_1(S^3-K)$ be its knot group.
We always assume that $m\ne 0$ and $n\ne 0$, unless specified otherwise.

\begin{proposition}
The knot group $G$ admits a presentation
\[
G=\langle x,y \mid w^nx=yw^n \rangle,
\]
where
$x$ and $y$ are meridians and $w=(xy^{-1})^m(x^{-1}y)^m$.
Furthermore, the longitude $\mathcal{L}$ is given as $\mathcal{L}=w_*^nw^n$, where
$w_*=(yx^{-1})^m(y^{-1}x)^m$ is obtained from $w$ by reversing the order of letters.
\end{proposition}

This is slightly different from that in \cite[Proposition 1]{HS},
but both are isomorphic.

\begin{proof}
We use a surgery diagram of $K$ as illustrated in Figure \ref{fig:pi1},
where $1/m$--surgery and $-1/n$--surgery are performed
along the second and third components, respectively.
Let $\mu_i$ and $\lambda_i$ be the meridian and longitude of the $i$th component.

\begin{figure}[ht]
\includegraphics*[scale=0.6]{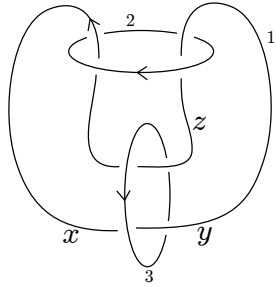}
\caption{A surgery diagram of $K$}\label{fig:pi1}
\end{figure}

First, $y=\mu_3^{-1}x\mu_3$, $z=\mu_2^{-1}y\mu_2$, $\lambda_2=x^{-1}y$ and $\lambda_3=yz^{-1}$.
By $1/m$--surgery on the second component yields a relation $\lambda_2^m\mu_2=1$, so $\mu_2=\lambda_2^{-m}$.
Similarly, $-1/n$--surgery on the third component gives $\mu_3=\lambda_3^{n}$. 
Thus
\begin{eqnarray*}
\lambda_3&=&y\mu_2^{-1}y^{-1}\mu_2=y\lambda_2^{m}y^{-1}\lambda_2^{-m}\\
&=&y(x^{-1}y)^my^{-1}(y^{-1}x)^m
=(yx^{-1})^m(y^{-1}x)^m.
\end{eqnarray*}
Then the relation $y=\mu_3^{-1}x\mu_3$ gives
\begin{equation}\label{eq:r1}
y[(x^{-1}y)^m(xy^{-1})^m]^{n}=[(x^{-1}y)^m(xy^{-1})^m]^{n}y.
\end{equation}

Set $w=(xy^{-1})^m(x^{-1}y)^m$.
Since $[(x^{-1}y)^m(xy^{-1})^m]^{n}=(x^{-1}y)^mw^n(x^{-1}y)^{-m}$,
(\ref{eq:r1}) changes to
$y(x^{-1}y)^mw^n(x^{-1}y)^{-m}=(x^{-1}y)^mw^n(x^{-1}y)^{-m}x$.
By $y(x^{-1}y)^m=(yx^{-1})^my$, we have
$(yx^{-1})^myw^n(x^{-1}y)^{-m}=(x^{-1}y)^mw^n(x^{-1}y)^{-m}x$.
Thus,
\begin{eqnarray*}
yw^n&=& (xy^{-1})^m(x^{-1}y)^mw^n(x^{-1}y)^{-m}x(x^{-1}y)^m\\
    &=& w^{n+1}(x^{-1}y)^{-m}x(x^{-1}y)^m\\
    &=& w^n(xy^{-1})^m(x^{-1}y)^m(x^{-1}y)^{-m}x(x^{-1}y)^m\\
    &=& w^n(xy^{-1})^mx(x^{-1}y)^m\\
    &=& w^nx.
\end{eqnarray*}

Set $w_*=(yx^{-1})^m(y^{-1}x)^m$.
Then $\lambda_3=w_*$.
The longitude is $\mathcal{L}=\mu_3\mu_2\mu_3^{-1}\mu_2^{-1}=w_*^n\mu_2w_*^{-n}\mu_2^{-1}$.
We have
\begin{eqnarray*}
\mu_2w_*^{-n}\mu_2^{-1}&=&(x^{-1}y)^{-m}[(yx^{-1})^m(y^{-1}x)^m]^{-n}(x^{-1}y)^{m}\\
&=&(x^{-1}y)^{-m}[(x^{-1}y)^m(xy^{-1})^m]^n(x^{-1}y)^{m}\\
&=&[(xy^{-1})^m(x^{-1}y)^m]^n\\
&=&w^n.
\end{eqnarray*}
Thus $\mathcal{L}=w_*^nw^n$.
\end{proof}

To describe the Riley polynomial of $K$ in Section  \ref{sec:riley}, we prepare
two sequences of polynomials with a single variable $s$.

For non-negative integer $m$,
let $f_m\in \mathbb{Z}[s]$ be defined by the recursion
\begin{equation}\label{eq:f}
f_{m+2}-(s+2)f_{m+1}+f_m=0
\end{equation}
with initial conditions $f_0=1$ and $f_1=s+1$.
Also,
let $g_m\in \mathbb{Z}[s]$ be defined by 
the same recursion
\begin{equation}\label{eq:g}
g_{m+2}-(s+2)g_{m+1}+g_m=0
\end{equation}
with slightly different initial conditions $g_0=1$ and $g_1=s+2$.
We remark that $g_m$ is equivalent to the Chebyshev polynomial of the second kind.

\begin{lemma}\label{lem:fg-formula}
The closed formulas for $f_m$ and $g_m$ are
\[
f_m=\sum_{i=0}^m \binom{m+i}{m-i}s^i,\quad  g_m=\sum_{i=0}^m\binom{m+1+i}{m-i}s^i.
\]
In particular, all coefficients of $f_m$ and $g_m$ are positive integers, and
the degree of $f_m$ and $g_m$ is $m$.
Also, $f_m$ and $g_m$ are monic.
\end{lemma}

\begin{proof}
These easily follow from the inductive argument by using the recursive formulas.
\end{proof}

Set $f_{-m}=f_{m-1}$ for $m\ge 1$, and set $g_{-1}=0$ and $g_{-m}=-g_{m-2}$ for $m\ge 2$.
Thus we have defined $f_m$ and $g_m$ for any integer $m$.
In particular, the recursions (\ref{eq:f}) and (\ref{eq:g}) hold for all integers.

\begin{lemma}\label{lem:fg}
For any integer $m$, the polynomials $f_m$, $g_m$ satisfy the following relations.
\begin{itemize}
\item[(1)] $f_m+g_{m-1}=g_m$.
\item[(2)] $f_m+s g_m=f_{m+1}$.
\item[(3)] $f_m^2=s g_m g_{m-1}+1$.
\end{itemize}
\end{lemma}

\begin{proof}
These are easily proved by induction.
We prove (3) only here.
Clearly, it holds when $m=0$.
Assume $f_m^2=sg_mg_{m-1}+1$.
From (1) and (2),
\begin{eqnarray*}
f_{m+1}^2&=& (f_m+sg_m)^2\\
&=& f_m^2+2sf_mg_m+s^2g_m^2\\
&=& (sg_mg_{m-1}+1)+2sf_mg_m+s^2g_m^2\\
&=& (sg_mg_{m-1}+1)+2s(g_m-g_{m-1})g_m+s^2g_m^2\\
&=& (s^2+2s)g_m^2-sg_mg_{m-1}+1\\
&=& s((s+2)g_m-g_{m-1})g_m+1\\
&=& sg_{m+1}g_m+1.
\end{eqnarray*}
Similarly, $f_{m-1}^2=(f_{m}-sg_{m-1})^2=sg_{m-1}g_{m-2}+1$.
\end{proof}

\begin{lemma}\label{lem:fg-ineq}
If a positive real number is substituted to $s$, then
we have the following inequalities.
\begin{itemize}
\item[(1)] For $m>0$, $sg_{m-1}<4f_m$.
\item[(2)] For $m<0$, $-sg_{m-1}<4f_{m-1}$.
\item[(3)] For $m<0$, $-2sg_{m-1}<3f_{m-1}$.
\end{itemize}
\end{lemma}

\begin{proof}
(1) Since $sg_{m-1}=f_m-f_{m-1}$ by Lemma \ref{lem:fg},
this follows from $3f_m+f_{m-1}>0$.
(2) is equivalent to (1).
(3) From $sg_{m-1}=f_m-f_{m-1}$,
it follows from $2f_m+f_{m-1}>0$.
\end{proof}


\section{Riley polynomials}\label{sec:riley}

In this section, we calculate the Riley polynomial of $K$.
Let $s$ and $t$ be real numbers such that $s>0$ and $t>1$.
Let $\rho: G\to SL_2(\mathbb{R})$ be a representation of $G$
defined by
\[
\rho(x)=\begin{pmatrix}
\sqrt{t} & 1/\sqrt{t} \\
0 & 1/\sqrt{t}\\
\end{pmatrix},
\quad
\rho(y)=\begin{pmatrix}
\sqrt{t} & 0 \\
-s \sqrt{t} & 1/\sqrt{t}\\
\end{pmatrix}.
\]
By \cite{Ri}, $\rho$ gives a non-abelian representation
if $s$ and $t$ are a pair of solutions of the Riley polynomial.
Recall that $w=(xy^{-1})^m(x^{-1}y)^m$.
Let $W=\rho(w)$ and $z_{i,j}$ be the $(i,j)$-entry of $W^n$.
Then the Riley polynomial of $K$ is given by $\phi_K(s,t)=z_{1,1}+(1-t)z_{1,2}$.
(See also \cite{DHY}.)

Since $s$ and $t$ are limited to be positive real numbers in our setting, it is not obvious that
there exist solutions for Riley's equation $\phi_K(s,t)=0$.
However, this will be verified in Proposition \ref{prop:root} under some condition.
In fact, we can choose $t$ so that $t>1$ and
$s+2<t+1/t<s+2+4/(sg_{m-1}^2)$ for any $s>0$.
We temporarily assume that $s$ and $t$ are chosen to satisfy $\phi_K(s,t)=0$.
See Example \ref{ex:riley} when $n=\pm 1$.

\begin{proposition}\label{prop:w-entry}
For $W=\rho(w)$, we have
\[
W=\begin{pmatrix}
f_m^2-st g_{m-1}^2 & f_{m-1}g_{m-1}-\frac{f_mg_{m-1}}{t}\\
sf_mg_{m-1}-stf_{m-1}g_{m-1} & f_{m-1}^2-\frac{s}{t}g_{m-1}^2
\end{pmatrix}.
\]
\end{proposition}

\begin{proof}
We prove by induction on $m$.
(For the knot $K(m,n)$, we assume $m\ne 0$.
However, this proposition holds even for $m=0$.)

If $m=0$, then $w=1$, so $W$ is the identity matrix.
It is easy to check that the claim holds.

Assume the conclusion for $m$.
Note 
\[
\rho(xy^{-1})= 
\begin{pmatrix}
s+1 & 1 \\
s & 1
\end{pmatrix},\ 
\rho(x^{-1}y)=\begin{pmatrix}
s+1 & -1/t \\
-st & 1
\end{pmatrix}.
\]
Calculate the product
\[
\begin{pmatrix}
s+1 & 1 \\
s & 1
\end{pmatrix}
\begin{pmatrix}
f_m^2-st g_{m-1}^2 & f_{m-1}g_{m-1}-\frac{f_mg_{m-1}}{t}\\
sf_mg_{m-1}-stf_{m-1}g_{m-1} & f_{m-1}^2-\frac{s}{t}g_{m-1}^2
\end{pmatrix} 
\begin{pmatrix}
s+1 & -1/t \\
-st & 1
\end{pmatrix}.
\]
By using Lemma \ref{lem:fg},
each entry is identified as desired.
For example, the $(1,1)$-entry
is
\begin{eqnarray*}
(f_m+s(f_m+g_{m-1}))^2-st(f_{m-1}+g_{m-1}+sg_{m-1})^2&=& (f_m+sg_m)^2-st(f_{m-1}+(s+1)g_{m-1})^2\\
&=& f_{m+1}^2-st(f_m+g_{m-1})^2\\
&=&f_{m+1}^2-stg_m^2.
\end{eqnarray*}

Similarly,
\[
\rho(yx^{-1})=\begin{pmatrix}
1 & -1\\
-s & s+1
\end{pmatrix},\ 
\rho(y^{-1}x)=\begin{pmatrix}
1 &  1/t\\
st & s+1
\end{pmatrix}.
\]
Then, calculating the product 
\[
\begin{pmatrix}
1 & -1 \\
-s & s+1
\end{pmatrix}
\begin{pmatrix}
f_m^2-st g_{m-1}^2 & f_{m-1}g_{m-1}-\frac{f_mg_{m-1}}{t}\\
sf_mg_{m-1}-stf_{m-1}g_{m-1} & f_{m-1}^2-\frac{s}{t}g_{m-1}^2
\end{pmatrix} 
\begin{pmatrix}
1 & 1/t \\
st & s+1
\end{pmatrix}
\]
gives the conclusion again.
\end{proof}

Let $\lambda_\pm\in \mathbb{C}$ be the eigenvalues of $W$.

\begin{lemma}\label{lem:eigenvalue}
$\lambda_+\ne\lambda_-$.
\end{lemma}

\begin{proof}
Since $\det W=1$, $\lambda_+\lambda_-=1$.
If $\lambda_+=\lambda_-$, then $\lambda_+=\lambda_-=\pm 1$.
Then $\mathrm{tr}W=\pm 2$.

On the other hand, 
$\mathrm{tr}W=f_m^2+f_{m-1}^2-s(t+1/t)g_{m-1}^2=s(s+2-t-1/t)g_{m-1}^2+2$ by Proposition \ref{prop:w-entry} and Lemma \ref{lem:fg}.
Since $g_{m-1}\ne 0$, $s>0$ and $s+2<t+1/t$,
$\mathrm{tr}W<2$.
If $\mathrm{tr}W=-2$, then
$s(t+1/t-s-2)g_{m-1}^2=4$, so $t+1/t=s+2+4/(sg_{m-1}^2)$.
But this is impossible by our choice of $t$, remarked before Proposition \ref{prop:w-entry}.
Hence $\lambda_+\ne \lambda_-$.
\end{proof}

Let $w_{i,j}$ be the $(i,j)$-entry of $W$, and
let $P=\begin{pmatrix}
w_{1,2} & w_{1,2} \\
\lambda_+-w_{1,1} & \lambda_--w_{1,1}
\end{pmatrix}$.

\begin{lemma}\label{lem:w12}
$w_{1,2}\ne 0$.
\end{lemma}

\begin{proof}
Assume $w_{1,2}=0$.
By Proposition \ref{prop:w-entry},
$w_{1,2}=f_{m-1}g_{m-1}-f_mg_{m-1}/t$.
Since $m\ne 0$, $g_{m-1}\ne 0$.
Hence $f_m=tf_{m-1}$.

From the recursion (\ref{eq:f}),
$tf_{m-1}-(s+2)f_{m-1}+f_{m-2}=0$.
Since $f_{m-1}\ne 0$,
we have $t=s+2-f_{m-2}/f_{m-1}$.
Thus $t<s+2$, contradicting our choice of $t$ again.
\end{proof}

The next will be used in Section \ref{sec:longi}.

\begin{lemma}\label{lem:w21}
$w_{2,1}=-st w_{1,2}$.
Hence $w_{2,1}\ne 0$.
\end{lemma}

\begin{proof}
This immediately follows from Proposition \ref{prop:w-entry}
and Lemma \ref{lem:w12}.
\end{proof}

By Lemmas \ref{lem:eigenvalue} and \ref{lem:w12},
$\det P=-w_{1,2}(\lambda_+-\lambda_-)\ne 0$.
A direct calculation shows
$P^{-1}WP=\begin{pmatrix}
\lambda_+  & 0\\
0 & \lambda_-
\end{pmatrix}$.

Set $\tau=\mathrm{tr}W$.
For any positive integer $k$,
let $\tau_k=\tau^{k-1}+\tau^{k-3}+\dots+\tau^{3-k}+\tau^{1-k}$.
If $\lambda_+\ne \lambda_-$, then $\tau_k=(\lambda_+^k-\lambda_-^k)/(\lambda_+-\lambda_-)$.
Let $\tau_0=0$ and $\tau_{-k}=-\tau_k$.
Then a recursion
$\tau_{k+1}-\tau\tau_k+\tau_{k-1}=0$
holds for any integer $k$.

\begin{lemma}\label{lem:wn}
For $W^n=\rho(w^n)$, we have
\[
W^n=
\begin{pmatrix}
w_{1,1}\tau_n-\tau_{n-1} & w_{1,2}\tau_n \\
w_{2,1}\tau_n & \tau_{n+1}-w_{1,1}\tau_n
\end{pmatrix}.
\]
\end{lemma}

\begin{proof}
When $n=\pm 1$, it is easy to see the conclusion.
Let $|n|>1$.
This easily follows from $W^n=P
\begin{pmatrix}
\lambda_+^n  & 0\\
0 & \lambda_-^n
\end{pmatrix}
P^{-1}$.
\end{proof}

\begin{proposition}\label{prop:riley}
The Riley polynomial of $K$ is
\[
\phi_K(s,t)=(\tau_{n+1}-\tau_n)+(s+2-t-1/t)f_{m-1}g_{m-1}\tau_n.
\]
\end{proposition}

\begin{proof}
By \cite{Ri} (see also \cite[p.309]{DHY}),
the Riley polynomial is $\phi_K(s,t)=z_{1,1}+(1-t)z_{1,2}$,
where $z_{i,j}$ is the $(i,j)$-entry of $W^n$.
From Lemma \ref{lem:wn} and the recursive formula for $\tau_k$,
\begin{eqnarray*}
\phi_K(s,t)&=&(w_{1,1}\tau_n-\tau_{n-1})+(1-t)w_{1,2}\tau_n\\
&=&\tau_{n+1}-(\mathrm{tr}W)\tau_n+(w_{1,1}+(1-t)w_{1,2})\tau_n\\
&=&(\tau_{n+1}-\tau_n)+(w_{1,1}+(1-t)w_{1,2}+1-\mathrm{tr}W)\tau_n.
\end{eqnarray*}
By Proposition \ref{prop:w-entry}, $\mathrm{tr}W=f_m^2+f_{m-1}^2-s(t+1/t)g_{m-1}^2$.
Thus we have
\[
w_{1,1}+(1-t)w_{1,2}+1-\mathrm{tr}W = (2-t-1/t)f_{m-1}g_{m-1}+1-f_{m-1}^2+sg_{m-1}^2.
\]
Since
\begin{eqnarray*}
1-f_{m-1}^2+sg_{m-1}^2 &=& -sg_{m-1}g_{m-2}+sg_{m-1}^2\\
&=& sg_{m-1}(g_{m-1}-g_{m-2})\\
&=& sg_{m-1}f_{m-1},
\end{eqnarray*}
we obtain
$\phi_K(s,t)=(\tau_{n+1}-\tau_n)+(s+2-t-1/t)f_{m-1}g_{m-1}\tau_n$.
\end{proof}

For convenience, we introduce a variable $T=t+1/t$.
Then the Riley polynomial of $K$ is
expressed as
$\phi_K(s,T)=(\tau_{n+1}-\tau_n)+(s+2-T)f_{m-1}g_{m-1}\tau_n$.

We remark that the Riley polynomial of $K$ is also described by \cite{MT}
in a different form.

\begin{example}\label{ex:riley}
If $n=1$, then
\begin{eqnarray*}
\phi_K(s,T)&=&(\tau_2-\tau_1)+(s+2-T)f_{m-1}g_{m-1}\tau_1\\
&=& (\mathrm{tr}W-1)+(s+2-T)f_{m-1}g_{m-1}\\
&=& s(s+2-T)g_{m-1}^2+1+(s+2-T)f_{m-1}g_{m-1}\\
&=& (s+2-T)g_{m-1}(sg_{m-1}+f_{m-1})+1\\
&=& (s+2-T)g_{m-1}f_m+1.
\end{eqnarray*}
Thus Riley's equation $\phi_K(s,T)=0$ has
the unique solution $T=s+2+1/(f_{m}g_{m-1})$ for any $s>0$.
If $m>0$, then $T>s+2>2$, because $f_m>0$ and $g_{m-1}>0$.
Hence we have a real solution $t=(T+\sqrt{T^2-4})/2>1$.
By Lemma \ref{lem:fg-ineq}(1),
we also have $s+2<T<s+2+4/(sg_{m-1}^2)$.

If $n=-1$, then $\phi_K(s,T)=1+(T-s-2)f_{m-1}g_{m-1}$.
Hence the equation $\phi_K(s,T)=0$ also has the unique solution
$T=s+2-1/(f_{m-1}g_{m-1})$.
If $m<0$, then $T>s+2>2$, so there exists a real solution $t>1$ as above.
Again, $s+2<T<s+2+4/(sg_{m-1}^2)$ by Lemma \ref{lem:fg-ineq}(2).
\end{example}

\section{Solutions of Riley's equation}

In this section, we examine when
Riley's equation $\phi_K(s,T)=0$ has a pair of real solutions $(s,T)$.
In fact, we can choose $T$ satisfying
$s+2+c/s<T<s+2+d/s$, where $c$ and $d$ are constants depending only $n$, for any $s>0$, unless $n=\pm 1$.

Let $k$ be a positive integer.
For $z=e^{i\theta}$ $(0\le\theta\le\pi)$,
set $\mathcal{T}_k(z)=z^{k-1}+z^{k-3}+\dots+z^{3-k}+z^{1-k}$.
If $z\ne \pm 1$, then $\mathcal{T}_k(z)=(z^k-z^{-k})/(z-z^{-1})$.
Define $\mathcal{T}_0=0$ and $\mathcal{T}_{-k}(z)=-\mathcal{T}_k(z)$.
Since $\mathcal{T}_k(z)$ is symmetric for $z$ and $z^{-1}$, it can be expanded as a polynomial
of $z+z^{-1}$.
Furthermore, a recursive relation
\[
\mathcal{T}_{k+1}(z)-(z+z^{-1})\mathcal{T}_k(z)+\mathcal{T}_{k-1}(z)=0
\]
holds.
Also, $\mathcal{T}_k(1)=k$ 
and $\mathcal{T}_k(-1)=(-1)^{k-1}k$
for any integer $k$.

\begin{lemma}\label{lem:tau}
\begin{itemize}
\item[(1)] Let $k\ge 1$. Then,
$\mathcal{T}_{k}(e^{\frac{\pi}{2k+1}i})=\mathcal{T}_{k+1}(e^{\frac{\pi}{2k+1}i})$.
This value is positive.
\item[(2)] Let $m\ge 2$. Then,
$\mathcal{T}_{k}(e^{\frac{3\pi}{2k+1}i})=\mathcal{T}_{k+1}(e^{\frac{3\pi}{2k+1}i})$.
This value is negative.
\end{itemize}
\end{lemma}

\begin{proof}
(1) Let $z=e^{\frac{\pi}{2k+1}i}$.
Then the fact that $z^{2k+1}=-1$ immediately implies
$\mathcal{T}_k(z)=\mathcal{T}_{k+1}(z)$.
A direct calculation shows
\[
\mathcal{T}_k(z)=\frac{\sin \frac{k\pi}{2k+1}}{\sin\frac{\pi}{2k+1}}>0.
\]

(2) Similarly, set $z=e^{\frac{3\pi}{2k+1}i}$.
Then $z^{2k+1}=-1$ holds again.
Hence we have $\mathcal{T}_k(z)=\mathcal{T}_{k+1}(z)$,
and
\[
\mathcal{T}_k(z)=\frac{\sin \frac{3k\pi}{2k+1}}{\sin\frac{3\pi}{2k+1}}<0.
\]
\end{proof}

Now, fix an $s>0$. 
We introduce a function $\Phi: [s+2,s+2+4/(sg_{m-1}^2)]\to\mathbb{R}$
by
\begin{equation}\label{eq:riley-P}
\Phi(T)=(\mathcal{T}_{n+1}(z)-\mathcal{T}_n(z))+(s+2-T)f_{m-1}g_{m-1}\mathcal{T}_n(z),
\end{equation}
where
$z=(\tau+i\sqrt{4-\tau^2})/2$ with 
$\tau=s(2+s-T)g_{m-1}^2+2$.

Note that $-2\le \tau\le 2$.
We will seek a solution $T$ for $\Phi(T)=0$ satisfying
$s+2<T<s+2+4/(sg_{m-1}^2)$, because
it gives a pair of solutions $(s,T)$ for Riley's equation $\phi_K(s,T)=0$.

\begin{proposition}\label{prop:root}
Suppose $n\ne \pm 1$.
For any $s>0$, Riley's equation $\phi_K(s,T)=0$ has a solution $T$ satisfying
$s+2+c/(sg_{m-1}^2)<T<s+2+d/(sg_{m-1}^2)$, where $c$ and $d$ are constants in $(0,4)$ depending only on $n$.
In particular, $\phi_K(s,t)=0$ has a solution $t>1$ for any $s>0$. 
\end{proposition}

\begin{proof}
Suppose $n>1$.
By Lemma \ref{lem:tau},
\[
\mathcal{T}_{n+1}(e^{\frac{\pi}{2n+1}i})=\mathcal{T}_n(e^{\frac{\pi}{2n+1}i}),\quad
\mathcal{T}_{n+1}(e^{\frac{3\pi}{2n+1}i})=\mathcal{T}_n(e^{\frac{3\pi}{2n+1}i}).
\]
Let $c=2-2\cos\frac{\pi}{2n+1}$ and $c'=2-2\cos\frac{3\pi}{2n+1}$.
Then $c,c'\in (0,4)$ and $c<c'$.

If $T=s+2+c/(sg_{m-1}^2)$ (resp. $s+2+c'/(sg_{m-1}^2$) then $\tau=2-c$ (resp. $2-c'$).
Thus
\begin{eqnarray*}
\Phi(s+2+\frac{c}{sg_{m-1}^2})&=&-\frac{cf_{m-1}}{sg_{m-1}}\mathcal{T}_n(e^{\frac{\pi}{2n+1}i}),\\
\Phi(s+2+\frac{c'}{sg_{m-1}^2})&=&-\frac{c'f_{m-1}}{sg_{m-1}}\mathcal{T}_n(e^{\frac{3\pi}{2n+1}i}).
\end{eqnarray*}
By Lemma \ref{lem:tau}, these values have distinct signs.
We remark that $\Phi(T)$ is a polynomial function of $T$, so it is continuous.
Thus
if $n>1$, we have a solution $T$ for $\Phi(T)=0$, satisfying $s+2+c/(sg_{m-1}^2)<T<s+2+c'/(sg_{m-1}^2)$,
from the Intermediate-Value Theorem.
Since $T>2$, $t+1/t=T$ has a real solution for $t$.
If we choose $t=(T+\sqrt{T^2-4})/2$, then $t>1$.

Suppose $n<-1$.   Set $l=|n|$.
If $l>2$, then 
set $d=2-2\cos\frac{\pi}{2l-1}$
and $d'=2-2\cos\frac{3\pi}{2l-1}$.
Then $d,d'\in (0,4)$ and $d<d'$.
As before,
\[
\mathcal{T}_{l-1}(e^{\frac{\pi}{2l-1}i})=\mathcal{T}_l(e^{\frac{\pi}{2l-1}i}),\quad
\mathcal{T}_{l-1}(e^{\frac{3\pi}{2l-1}i})=\mathcal{T}_l(e^{\frac{3\pi}{2l-1}i}).
\]
by Lemma \ref{lem:tau}.
Thus
\begin{eqnarray*}
\Phi(s+2+\frac{d}{sg_{m-1}^2})
&=& \frac{d\,f_{m-1}}{sg_{m-1}^2}\mathcal{T}_l(e^{\frac{\pi}{2l-1}i}),\\
\Phi(s+2+\frac{d'}{sg_{m-1}^2})
&=& \frac{d'f_{m-1}}{sg_{m-1}^2}\mathcal{T}_l(e^{\frac{3\pi}{2l-1}i}).
\end{eqnarray*}
Since these values have distinct signs,
we have a solution $T$ with $s+2+d/(sg_{m-1}^2)<T<s+2+d'/(sg_{m-1}^2)$, if $l>2$, as before.

When $l=2$, 
we have
\[
\Phi(s+2+\frac{1}{sg_{m-1}^2})=\frac{f_{m-1}}{sg_{m-1}},\quad  \Phi(s+2+\frac{2}{sg_{m-1}^2})=-1,
\]
\[
\Phi(s+2+\frac{3}{sg_{m-1}^2})=-2-\frac{3f_{m-1}}{sg_{m-1}}.
\]
If $m>0$, then $\Phi(s+2+1/(sg_{m-1}^2))>0$, since $f_{m-1},g_{m-1}>0$.
Therefore there exists a solution $T$ with $s+2+1/(sg_{m-1}^2)<T<s+2+2/(sg_{m-1}^2)$.
If $m<0$, $\Phi(s+2+3/(sg_{m-1}^2))>0$, since $3f_{m-1}>-2sg_{m-1}$ by Lemma \ref{lem:fg-ineq}(3).
Thus there exists a solution $T$ with $s+2+2/(sg_{m-1}^2)<T<s+2+3/(sg_{m-1}^2)$
\end{proof}

\section{Longitudes}\label{sec:longi}

For $s>0$,
let
$\rho_s: G\to SL_2(\mathbb{R})$ be the representation defined by the correspondence
\begin{equation}\label{eq:rho}
\rho_s(x)=
\begin{pmatrix}
\sqrt{t} & 0 \\
0 & \frac{1}{\sqrt{t}}
\end{pmatrix},\quad
\rho_s(y)=
\begin{pmatrix}
\frac{t-s-1}{\sqrt{t}-\frac{1}{\sqrt{t}}} & \frac{s}{(\sqrt{t}-\frac{1}{\sqrt{t}})^2}-1 \\
-s & \frac{s+1-\frac{1}{t}}{\sqrt{t}-\frac{1}{\sqrt{t}}}
\end{pmatrix}.
\end{equation}

For $Q=\begin{pmatrix}
t-1 & 1\\
0 & \sqrt{t}-\frac{1}{\sqrt{t}}
\end{pmatrix}$,
\[
Q^{-1} \rho_s(x) Q=\begin{pmatrix} 
\sqrt{t} & \frac{1}{\sqrt{t}} \\
0 & \frac{1}{\sqrt{t}}
\end{pmatrix},\quad 
Q^{-1} \rho_s(y) Q=\begin{pmatrix}
\sqrt{t} & 0\\
-s\sqrt{t} & \frac{1}{\sqrt{t}}
\end{pmatrix}.
\]
Therefore, $\rho_s$ is conjugate with $\rho$ defined in Section \ref{sec:riley}.
This implies that if $s$ and $t$ satisfy Riley's equation $\phi_K(s,t)=0$ 
then $\rho_s$ gives a representation of $G$ as well as $\rho$.
In this section, we examine the image of the longitude $\mathcal{L}$ of $G$
under $\rho_s$.

Throughout the section,
let $U=\rho_s(w)$ and $u_{i,j}$ be its $(i,j)$-entry, and
let $v_{i,j}$ be the entries of $U^n$. 
Also, set $\sigma=\frac{s(\sqrt{t}-\frac{1}{\sqrt{t}})^2}{(\sqrt{t}-\frac{1}{\sqrt{t}})^2-s}$.

\begin{lemma}\label{lem:ww*}
For $w_*^n$, we have
$\rho_s(w_*^n)=\begin{pmatrix}
v_{1,1} & \frac{v_{2,1}}{\sigma}\\
v_{1,2}\sigma & v_{2,2}
\end{pmatrix}$.
\end{lemma}

\begin{proof}
By a direct calculation,
\[
\rho_s(xy^{-1})=\begin{pmatrix}
\frac{t-1+st}{t-1} & \frac{s\sqrt{t}}{\sigma} \\
\frac{s}{\sqrt{t}} & \frac{t-1-s}{t-1}
\end{pmatrix}, \quad
\rho_s(y^{-1}x)=\begin{pmatrix}
\frac{t-1+st}{t-1} & \frac{s}{\sqrt{t}\sigma} \\
s\sqrt{t} & \frac{t-1-s}{t-1}
\end{pmatrix},
\]
\[
\rho_s(x^{-1}y)=\begin{pmatrix}
\frac{t-1-s}{t-1} & -\frac{s}{\sqrt{t}\sigma} \\
-s\sqrt{t} & \frac{t-1+st}{t-1}
\end{pmatrix}, \quad
\rho_s(yx^{-1})=\begin{pmatrix}
\frac{t-1-s}{t-1} & -\frac{s\sqrt{t}}{\sigma} \\
-\frac{s}{\sqrt{t}} & \frac{t-1+st}{t-1}
\end{pmatrix}.
\]
Thus we see that the $(1,2)$-entry of $\rho_s(y^{-1}x)$ is
the $(2,1)$-entry of $\rho_s(xy^{-1})$ divided by $\sigma$,
the $(2,1)$-entry of $\rho_s(y^{-1}x)$ is
the $(1,2)$-entry of $\rho_s(xy^{-1})$ multiplied by $\sigma$,
and the others of $\rho_s(y^{-1}x)$ coincide with those of $\rho_s(xy^{-1})$.
The same relation between entries holds for $\rho_s(x^{-1}y)$ and $\rho_s(yx^{-1})$.

In general, such a relation is preserved under the matrix multiplication;
\[
\begin{pmatrix}
a & b \\
c & d
\end{pmatrix}
\begin{pmatrix}
p & q \\
r & s
\end{pmatrix}
=
\begin{pmatrix}
ap+br & aq+bs \\
cp+dr & cq+ds
\end{pmatrix},
\]
\[
\begin{pmatrix}
p & \frac{r}{\sigma} \\
q\sigma & s
\end{pmatrix}
\begin{pmatrix}
a & \frac{c}{\sigma} \\
b\sigma & d
\end{pmatrix}
=
\begin{pmatrix}
ap+br & \frac{cp+dr}{\sigma} \\
(aq+bs)\sigma & cq+ds
\end{pmatrix}.
\]
Thus we can confirm that the same relation holds for $\rho_s(w^n)$ and $\rho_s(w_*^n)$.
\end{proof}

\begin{proposition}\label{prop:longi}
For the longitude $\mathcal{L}$ of $G$,
the matrix $\rho_s(\mathcal{L})$ is diagonal, and
the $(1,1)$-entry of $\rho_s(\mathcal{L})$ is a positive real number.
\end{proposition}

\begin{proof}
The first assertion follows from the facts that
for a meridian $x$, 
$\rho_s(x)$ is diagonal but $\rho_s(x)\ne \pm I$ and that $x$ and $\mathcal{L}$ commute.

Since $\mathcal{L}=w_*^n w^n$, Lemma \ref{lem:ww*} implies that
\begin{eqnarray*}
\rho_s(\mathcal{L})=\rho_s(w_*^n)\rho_s(w^n)&=&
\begin{pmatrix}
v_{1,1} & \frac{v_{2,1}}{\sigma} \\
v_{1,2}\sigma & v_{2,2}
\end{pmatrix}
\begin{pmatrix}
v_{1,1} & v_{1,2} \\
v_{2,1} & v_{2,2}
\end{pmatrix}
\\
&=&
\begin{pmatrix}
v_{1,1}^2+\frac{v_{2,1}^2}{\sigma} & v_{1,1}v_{1,2}+\frac{v_{2,1}v_{2,2}}{\sigma}\\ 
v_{1,1}v_{1,2}\sigma+v_{2,1}v_{2,2} & v_{1,2}^2\sigma+v_{2,2}^2
\end{pmatrix}.
\end{eqnarray*}

Hence the $(1,1)$-entry is $v_{1,1}^2+v_{2,1}^2/\sigma$, which
is positive, because
$s>0$ and $(\sqrt{t}-1/\sqrt{t})^2-s=T-s-2>0$ from Proposition \ref{prop:root}.
\end{proof}

\begin{remark}\label{rem:longi}
Since $\rho_s(\mathcal{L})$ is diagonal,
we also obtain an equation $v_{1,1}v_{1,2}\sigma+v_{2,1}v_{2,2}=0$.
This will be used in the proof of Lemma \ref{lem:bs}.
\end{remark}


For $W=\rho(w)$, recall that $w_{i,j}$ is its entry.

\begin{lemma}\label{lem:u-entry}
For $U=\rho_s(w)$,
\begin{eqnarray*}
u_{1,1}&=& w_{1,1}+\frac{w_{2,1}}{t-1}, \quad u_{1,2}=\sqrt{t}\left(w_{1,2}-\frac{w_{1,1}}{t-1}\right)+\frac{\sqrt{t}}{t-1}\left(w_{2,2}-\frac{w_{2,1}}{t-1}\right),\\
u_{2,1}&=& \frac{w_{2,1}}{\sqrt{t}}, \quad u_{2,2}=w_{2,2}-\frac{w_{2,1}}{t-1}.
\end{eqnarray*}
\end{lemma}

\begin{proof}
This immediately follows by calculating the product $U=QWQ^{-1}$.
\end{proof}

By Lemma \ref{lem:wn},
we have 
\[
W^n=\begin{pmatrix}
w_{1,1}\tau_n-\tau_{n-1} & w_{1,2}\tau_n\\
w_{2,1}\tau_n & \tau_{n+1}-w_{1,1}\tau_n
\end{pmatrix}.
\]

\begin{lemma}\label{lem:vij}
For $U^n=\rho_s(w^n)$, we have
\begin{eqnarray*}
v_{1,1}&=& u_{1,1}\tau_n-\tau_{n-1},\quad
v_{1,2}= u_{1,2}\tau_n,\\
v_{2,1}&=& u_{2,1}\tau_n\quad
v_{2,2}= \tau_{n+1}-u_{1,1}\tau_n.
\end{eqnarray*}
\end{lemma}

\begin{proof}
Calculate the product $U^n=QW^nQ^{-1}$.
Then $v_{1,1}=(w_{1,1}+w_{2,1}/(t-1))\tau_n-\tau_{n-1}=u_{1,1}\tau_n-\tau_{n-1}$
by Lemma \ref{lem:u-entry}.
For $v_{1,2}$,
we have
\[
\begin{split}
v_{1,2}&= \frac{\sqrt{t}}{(t-1)^2}\Bigl((t-1)\tau_{n-1}-(t-1)w_{1,1}\tau_n+(t-1)^2w_{1,2}\tau_n\\
&\quad +(t-1)\tau_{n+1}-(t-1)w_{1,1}\tau_n-w_{2,1}\tau_n\Bigr).
\end{split}
\]
Recall that the Riley polynomial is $\phi_K(s,t)=(w_{1,1}\tau_n-\tau_{n-1})+(1-t)w_{1,2}\tau_n$.
(See the proof of Proposition \ref{prop:riley}.)
Since $\phi_K(s,t)=0$,
$\tau_{n-1}=w_{1,1}\tau_n+(1-t)w_{1,2}\tau_n$.
Hence
\begin{eqnarray*}
v_{1,2}&=& \frac{\sqrt{t}}{(t-1)^2}\Bigl((t-1)\tau_{n+1}-(w_{2,1}+(t-1)w_{1,1})\tau_n\Bigr)\\
&=&\frac{\sqrt{t}}{(t-1)^2}\Bigl(\bigl((t-1)w_{2,2}-w_{2,1}\bigr)\tau_n-(t-1)\tau_{n-1}\Bigr)\\
&=&\frac{\sqrt{t}}{t-1}\Bigl((w_{2,2}-\frac{w_{2,1}}{t-1})\tau_n-\tau_{n-1}\Bigr)\\
&=&\frac{\sqrt{t}}{t-1}(u_{2,2}\tau_n-\tau_{n-1}),
\end{eqnarray*}
by using the recursion $\tau_{n+1}-(w_{1,1}+w_{2,2})\tau_n+\tau_{n-1}=0$.
By Lemma \ref{lem:u-entry},
\[
u_{2,2}=\frac{t-1}{\sqrt{t}}\left(u_{1,2}-\sqrt{t}\left(w_{1,2}-\frac{w_{1,1}}{t-1}\right)\right).
\]
Substituting this and $\phi_K(s,t)=0$,
\begin{eqnarray*}
v_{1,2}&=& u_{1,2}\tau_n+\frac{\sqrt{t}}{t-1}\Bigl(w_{1,1}\tau_n-\tau_{n-1}+(1-t)w_{1,2}\tau_n\Bigr)\\
&=& u_{1,2}\tau_n.
\end{eqnarray*}

It is straightforward to check $v_{2,1}$ and $v_{2,2}$. We omit them.
\end{proof}

Let $B_s$ be the $(1,1)$-entry of the matrix $\rho_s(\mathcal{L})$.

\begin{lemma}\label{lem:bs}
$B_s=-u_{2,1}/(u_{1,2}\sigma)$.
\end{lemma}

\begin{proof}
As noted in Remark \ref{rem:longi}, $v_{1,1}v_{1,2}\sigma+v_{2,1}v_{2,2}=0$.
Since $\det U^n=v_{1,1}v_{2,2}-v_{1,2}v_{2,1}=1$,
we have
\begin{eqnarray*}
v_{1,2}B_s &=& v_{1,1}^2v_{1,2}+\frac{v_{1,2}v_{2,1}^2}{\sigma}\\
  &=& v_{1,1}^2v_{1,2}+\frac{v_{2,1}}{\sigma}(v_{1,1}v_{2,2}-1)\\       
  &=& v_{1,1}^2v_{1,2}+\frac{v_{1,1}}{\sigma}(-v_{1,1}v_{1,2}\sigma)-\frac{v_{2,1}}{\sigma}\\
  &=& -\frac{v_{2,1}}{\sigma}.
\end{eqnarray*}

Assume $v_{1,2}=0$.
Then $v_{2,1}=0$.
By Lemmas \ref{lem:u-entry} and \ref{lem:vij}, $v_{2,1}=u_{2,1}\tau_n=w_{2,1}\tau_n/\sqrt{t}$.
Since $w_{2,1}\ne 0$ by Lemma \ref{lem:w21}, we have $\tau_n=0$.
Recall that $\phi_K(s,t)=(\tau_{n+1}-\tau_n)+(s+2-t-1/t)f_{m-1}g_{m-1}\tau_n$ is zero.
Thus $\tau_{n+1}=0$.
Then the recursive formula for $\tau_k$ implies $\tau_{n-1}=0$.
In turn, all $\tau_k=0$.   But this is impossible, because $\tau_1=1$.
Hence $v_{1,2}\ne 0$, so $B_s=-v_{2,1}/(v_{1,2}\sigma)$.

By Lemma \ref{lem:vij}, $v_{1,2}=u_{1,2}\tau_n$ and $v_{2,1}=u_{2,1}\tau_n$.
Thus we have shown that $B_s=-u_{2,1}/(u_{1,2}\sigma)$.
\end{proof}

\begin{proposition}\label{prop:bs}
For the longitude $\mathcal{L}$, the $(1,1)$-entry $B_s$ of $\rho_s(\mathcal{L})$ is given as
\begin{equation}
B_s=\frac{-f_m+tf_{m-1}}{-f_{m-1}+tf_m}.
\end{equation}
\end{proposition}

\begin{proof}
From Lemma \ref{lem:bs}, $B_s=-u_{2,1}/(u_{1,2}\sigma)$.
Then Lemma \ref{lem:u-entry} and Proposition \ref{prop:w-entry},
\begin{eqnarray*}
u_{1,2}&=&\frac{\sqrt{t}}{t-1}\left(f_{m-1}^2-f_m^2+(st-\frac{s}{t})g_{m-1}^2\right)-\frac{(t-1)^2+st}{\sqrt{t}(t-1)^2}g_{m-1}(f_{m}-tf_{m-1}),\\
u_{2,1}&=&\frac{s}{\sqrt{t}}g_{m-1}(f_{m}-tf_{m-1}).
\end{eqnarray*}

For the first term of $u_{1,2}$, Lemma \ref{lem:fg} implies
\begin{eqnarray*}
f_{m-1}^2-f_m^2+(st-\frac{s}{t})g_{m-1}^2&=& (f_{m-1}-f_{m})(f_{m-1}+f_{m})+(t-\frac{1}{t})sg_{m-1}^2\\
&=& -sg_{m-1}(f_{m-1}+f_{m})+(t-\frac{1}{t})sg_{m-1}^2\\
&=& sg_{m-1}\left(-f_{m-1}-f_{m}+(t-\frac{1}{t})g_{m-1}\right)\\
&=& sg_{m-1}\left(-f_{m-1}-f_{m}+(t-\frac{1}{t})\frac{f_{m}-f_{m-1}}{s}\right)\\
&=& sg_{m-1}\left(\frac{t^2-1-st}{st}f_m-\frac{t^2-1+st}{st}f_{m-1}\right).
\end{eqnarray*}
Thus,
dividing $u_{1,2}$ by $sg_{m-1}/\sqrt{t}$ gives
\[
\frac{t}{t-1}\left(\frac{t^2-1-st}{st}f_m-\frac{t^2-1+st}{st}f_{m-1}\right)-\frac{(t-1)^2+st}{s(t-1)^2}(f_m-tf_{m-1}).
\]
The coefficient of $f_m$ is $t/\sigma$, and
that of $f_{m-1}$ is $-1/\sigma$.
Hence we have
\[
B_s=-\frac{u_{2,1}}{u_{1,2}\sigma}=\frac{-f_m+tf_{m-1}}{-f_{m-1}+tf_m}.
\]
\end{proof}


\section{Limits}

Let $r=p/q$ be a rational number, and let $K(r)$ denote
the resulting manifold by $r$-surgery on $K$.
In other words, $K(r)$ is obtained by 
attaching a solid torus $V$ to the knot exterior $E(K)$ along their boundaries so that
the loop $x^p\mathcal{L}^q$ bounds a meridian disk of $V$, where
$x$ and $\mathcal{L}$ are a meridian and longitude of $K$.

Our representation $\rho_s: G\to SL_2(\mathbb{R})$ induces
a homomorphism $\pi_1(K(r))\to SL_2(\mathbb{R})$
if and only if $\rho_s(x)^p\rho_s(\mathcal{L})^q=I$.
Since both of $\rho_s(x)$ and $\rho_s(\mathcal{L})$ are diagonal
(see (\ref{eq:rho}) and Proposition \ref{prop:longi}),
this is equivalent to the single equation
\begin{equation}\label{eq:slope}
A_s^p B_s^q=1,
\end{equation}
where $A_s$ and $B_s$ are the $(1,1)$-entries of $\rho_s(x)$ and $\rho_s(\mathcal{L})$, respectively.
We remark that $A_s=\sqrt{t}\ (>1)$ is a positive real number,
so is $B_s$ by Proposition \ref{prop:longi}.
Hence the equation (\ref{eq:slope}) is furthermore equivalent to the equation
\begin{equation}
-\frac{\log B_s}{\log A_s}=\frac{p}{q}.
\end{equation}

Let $g:(0,\infty)\to \mathbb{R}$ be a function defined by
\[
g(s)=-\frac{\log B_s}{\log A_s}.
\]

We will examine the image of $g$.

\begin{lemma}\label{lem:key-limit}
\begin{itemize}
\item[(1)] 
\[
\lim_{s\to+0}t=\begin{cases}
\infty & \text{if $n\ne \pm 1$},\\
\frac{2m+1+\sqrt{4m+1}}{2m} & \text{if $n=1$},\\
\frac{2m-1+\sqrt{1-4m}}{2m} & \text{if $n=-1$}.
\end{cases}
\]
\item[(2)] $\displaystyle\lim_{s\to\infty}t=\infty$.
\item[(3)] $\displaystyle\lim_{s\to\infty}(t-s)=2$.
\item[(4)] $\displaystyle\lim_{s\to\infty}\frac{t}{s}=1$.
\end{itemize}
\end{lemma}

\begin{proof}
(1) 
If $n=1$, then $T=s+2+1/(f_mg_{m-1})$ is the unique solution for $\phi_K(s,T)=0$
(see Example \ref{ex:riley}).
From Lemma \ref{lem:fg-formula}, we have $\lim_{s\to+0}f_m=1$ and $\lim_{s\to+0}g_{m-1}=m$.
Hence $\lim_{s\to +0}T=2+1/m$.
(Recall $m>0$ when $n=1$.)
Since $t=(T+\sqrt{T^2-4})/2$, $\lim_{s\to+0}t=(2m+1+\sqrt{4m+1})/(2m)$.

If $n=-1$, then $T=s+2-1/(f_{m-1}g_{m-1})$ by Example \ref{ex:riley}.
Recall $m<0$ when $n=-1$.  Then $\lim_{s\to+0}f_{m-1}=1$ and $\lim_{s\to+0}g_{m-1}=m$.
Thus $\lim_{s\to+0}T=2-1/m$, so 
$\lim_{s\to+0}t=(2m-1+\sqrt{1-4m})/(2m)$.

Assume $n\ne \pm 1$.
From Proposition \ref{prop:root}, we have $s+2+c/(sg_{m-1}^2)<T$, where $c$ is a positive constant.
Hence $\lim_{s\to+0}T=\lim_{s\to+0}t=\infty$.

(2) As $T>s+2$, $\lim_{s\to\infty}T=\lim_{s\to\infty}t=\infty$.

(3) Since $s+2<t+1/t<s+2+4/(sg_{m-1}^2)$, (2) implies $\lim_{s\to\infty}(t-s)=2$.

(4) From $s+2<T<s+2+4/(sg_{m-1}^2)$ again, we have $\lim_{s\to\infty}T/s=1$, which
implies $\lim_{s\to\infty}t/s=1$
\end{proof}

Let $F_k=t^{k-1}(-f_k+tf_{k-1})$ for any integer $k$.

\begin{lemma}\label{lem:f-limit}
If $0<k\le |m|$, then $\lim_{s\to\infty}F_k=1$.
\end{lemma}

\begin{proof}
First, $F_1=-f_1+tf_0=t-s-1$.
Hence $\lim_{s\to\infty}F_1=1$ by Lemma \ref{lem:key-limit}(3).

Let $n\ne\pm 1$.
By Proposition \ref{prop:root}, we have
\begin{equation}\label{eq:root}
s+2+\frac{c}{sg_{m-1}^2}<t+\frac{1}{t}<s+2+\frac{d}{sg_{m-1}^2},
\end{equation}
where $c$ and $d$ are constants depending on only $n$.
Multiplying $t^{k-1}f_{k-1}$ to (\ref{eq:root})
gives
\[
(s+2)f_{k-1}t^{k-1}+\frac{ct^{k-1}f_{k-1}}{sg_{m-1}^2}<t^kf_{k-1}+t^{k-2}f_{k-1}<(s+2)f_{k-1}t^{k-1}+\frac{dt^{k-1}f_{k-1}}{sg_{m-1}^2}.
\]
Since $(s+2)f_{k-1}=f_k+f_{k-2}$ by the recursion,
\begin{eqnarray*}
(t^{k-1}f_{k-2}-t^{k-2}f_{k-1})+\frac{ct^{k-1}f_{k-1}}{sg_{m-1}^2}&<&t^kf_{k-1}-t^{k-1}f_k\\
&<&(t^{k-1}f_{k-2}-t^{k-2}f_{k-1})+\frac{dt^{k-1}f_{k-1}}{sg_{m-1}^2}.
\end{eqnarray*}

Then we have
\begin{equation}
F_{k-1}+\frac{ct^{k-1}f_{k-1}}{sg_{m-1}^2}<F_k<F_{k-1}+\frac{dt^{k-1}f_{k-1}}{sg_{m-1}^2}.
\end{equation}
The degree of $f_{k-1}$ is $k-1$, 
but that of $g_{m-1}$ is $|m|-1$.
Hence Lemma \ref{lem:key-limit}(4) implies
\[
\lim_{s\to\infty}\frac{ct^{k-1}f_{k-1}}{sg_{m-1}^2}=
\lim_{s\to\infty}\frac{dt^{k-1}f_{k-1}}{sg_{m-1}^2}=0
\]
as long as $k\le |m|$.
Thus $\lim_{s\to\infty}F_k=1$ will follow from $\lim_{s\to\infty}F_{k-1}=1$.

Suppose $n=1$.
Riley's equation has the unique solution
\begin{equation}\label{eq:n1}
t+\frac{1}{t}=s+2+\frac{1}{f_mg_{m-1}}
\end{equation}
(see Example \ref{ex:riley}).

Multiplying $t^{k-1}f_{k-1}$ to (\ref{eq:n1})
gives
\[
t^kf_{k-1}+t^{k-2}f_{k-1}=(s+2)f_{k-1}t^{k-1}+\frac{t^{k-1}f_{k-1}}{f_mg_{m-1}}.
\]
Since $(s+2)f_{k-1}=f_k+f_{k-2}$,
\[
t^kf_{k-1}-t^{k-1}f_k=t^{k-1}f_{k-2}-t^{k-2}f_{k-1}+\frac{t^{k-1}f_{k-1}}{f_mg_{m-1}}.
\]
That is, $F_k=F_{k-1}+t^{k-1}f_{k-1}/(f_mg_{m-1})$.
Hence the fact that $\lim_{s\to\infty}F_{k-1}=1$ and $\lim_{s\to\infty}t^{k-1}f_{k-1}/(f_mg_{m-1})=0$
imply $\lim_{s\to\infty}F_k=1$.

Finally, suppose $n=-1$.
Riley's equation has the unique solution
\[
t+\frac{1}{t}=s+2-\frac{1}{f_{m-1}g_{m-1}}
\]
as in Example \ref{ex:riley}.
Multiplying $t^{k-1}f_{k-1}$ gives
$F_k=F_{k-1}-t^{k-1}f_{k-1}/(f_{m-1}g_{m-1})$.
Since $m<0$, $f_{m-1}$ and $g_{m-1}$
have degree $|m|$ and $|m|-1$, respectively.
Thus the fact that $\lim_{s\to\infty}F_{k-1}=1$ and $\lim_{s\to\infty}t^{k-1}f_{k-1}/(f_{m-1}g_{m-1})=0$
imply $\lim_{s\to\infty}F_k=1$, again.
\end{proof}

\begin{lemma}\label{lem:bslimit}
\begin{itemize}
\item[(1)] $\displaystyle\lim_{s\to+0}B_s=1$.
\item[(2)] $\displaystyle\lim_{s\to\infty}B_s\,t^{2m}=1$.
\end{itemize}
\end{lemma}

\begin{proof}
(1) By Proposition \ref{prop:bs},
\[
B_s=\frac{-f_m+tf_{m-1}}{-f_{m-1}+tf_m}.
\]
By Lemma \ref{lem:fg-formula}, $\lim_{s\to+0}f_m=\lim_{s\to+0}f_{m-1}=1$.
Thus Lemma \ref{lem:key-limit}(1) implies $\lim_{s\to+0}B_s=1$.

(2)  Let $m>0$.  We decompose $B_st^{2m}$ as
\begin{eqnarray*}
B_st^{2m}&=&t^{m-1}(-f_m+tf_{m-1})\cdot \frac{t^{m+1}}{-f_{m-1}+tf_m}\\
&=& F_m\cdot\frac{t^{m+1}}{-f_{m-1}+tf_m}.
\end{eqnarray*}
Since the degree of $f_k\ (k>0)$ is $k$ and $f_m$ is monic,
\[
\lim_{s\to\infty}\frac{t^{m+1}}{-f_{m-1}+tf_m}=1.
\]
Then
we have $\lim_{s\to\infty}B_st^{2m}=1$ by combined with Lemma \ref{lem:f-limit}.

Let $m<0$.  Set $l=-m>0$.
Recall $f_m=f_{l-1}$ and $f_{m-1}=f_{l}$.
We decompose $B_st^{2m}$ as
\[
B_st^{2m}=\frac{-f_{l-1}+tf_l}{t^{l+1}}\cdot \frac{1}{t^{l-1}(-f_l+tf_{l-1})}=\frac{-f_{l-1}+tf_l}{t^{l+1}}\cdot \frac{1}{F_l}.
\]
As before,
\[
\lim_{s\to\infty}\frac{-f_{l-1}+tf_l}{t^{l+1}}=1\ \text{and} \lim_{s\to\infty}F_l=1.
\]
Thus $\lim_{s\to\infty}B_st^{2m}=1$ again.
\end{proof}

\begin{proposition}\label{prop:g-image}
The image of $g$ contains an open interval 
$(0,4m)$ \textup{(}resp. $(4m,0)$\textup{)} if $m>0$ \textup{(}resp. $m<0$\textup{)}.
\end{proposition}

\begin{proof}
By Lemma \ref{lem:bslimit}(1),
$\lim_{s\to+0}\log B_s=0$.
Hence
\[
\lim_{s\to +0}g(s)=-\lim_{s\to +0}\frac{\log B_s}{\log A_s}=-\lim_{s\to+0}\frac{\log B_s}{\log\sqrt{t}}=0.
\]

Also, we have $\lim_{s\to\infty}(\log B_s+2m\log t)=0$ by Lemma \ref{lem:bslimit}(2).
Thus
\[
\lim_{s\to \infty}g(s)=-\lim_{s\to\infty}\frac{\log B_s}{\log A_s}=
-\lim_{s\to\infty}\frac{2\log B_s}{\log t}=4m.
\]
Hence the image of $g$ contains an interval $(0,4m)$ or $(4m,0)$, according as the sign of $m$.
\end{proof}


\section{Proof of Theorem}

The universal covering group $\widetilde{SL_2(\mathbb{R})}$
of $SL_2(\mathbb{R})$ 
can be described as
\[
\widetilde{SL_2(\mathbb{R})}=\{(\gamma,\omega)\mid \gamma\in \mathbb{C},|\gamma|<1,-\infty<\omega<\infty\},
\]
 (see \cite{B}).  Let $\chi:\widetilde{SL_2(\mathbb{R})}\to SL_2(\mathbb{R})$
be the covering projection.
Then $\ker \chi=\{(0,2m\pi)\mid m\in \mathbb{Z}\}$ is isomorphic to $\mathbb{Z}$.

Since the knot exterior $E(K)$ of $K$ satisfies $H^2(E(K);\mathbb{Z})=0$,
any $\rho_s: G\to SL_2(\mathbb{R})$ lifts to a representation
$\tilde{\rho}: G\to\widetilde{SL_2(\mathbb{R})}$ (\cite{G}).
Moreover, any two lifts $\tilde{\rho}$ and $\tilde{\rho}'$ are
related as follows:
\[
\tilde{\rho}'(g)=h(g)\tilde{\rho}(g),
\]
where $h:G\to \ker \chi\subset\widetilde{SL_2(\mathbb{R})}$.
Since $\ker \chi$ is abelian,
the homomorphism $h$ factors through $H_1(E(K))$, so
it is determined only by the value $h(x)$ of a meridian $x$ (see \cite{Kh}).

\begin{lemma}\label{lem:key}
Let $\tilde{\rho}: G\to \widetilde{SL_2(\mathbb{R})}$ be a lift of $\rho_s$.
Then replacing $\tilde{\rho}$ by a representation
$\tilde{\rho}'=h\cdot \tilde{\rho}$ for some $h:G\to \widetilde{SL_2(\mathbb{R})}$,
we can suppose that $\tilde{\rho}(\pi_1(\partial E(K)))$ is contained in the subgroup $(-1,1)\times \{0\}$ of $\widetilde{SL_2(\mathbb{R})}$.
\end{lemma}

\begin{proof}
This is proved in \cite[Section 7]{HT1} for twist knots.
Since our knot $K$ has genus one, the argument works without any change.
\end{proof}

\begin{proof}[Proof of Theorem \ref{thm:main}]
Suppose that $m,n>0$.
Let $r=p/q\in (0,4m)$.
By Proposition \ref{prop:g-image},
we can find $s$ so that $g(s)=r$.
Choose a lift $\tilde{\rho}$ of $\rho_s$ so that
$\tilde{\rho}(\pi_1(\partial E(K)))\subset (-1,1)\times\{0\}$ (Lemma \ref{lem:key}).
Then $\rho_s(x^p\mathcal{L}^q)=I$, so $\chi(\tilde{\rho}(x^p\mathcal{L}^q))=I$.
This means that $\tilde{\rho}(x^p\mathcal{L}^q)$ lies in $\ker\chi=\{(0,2m\pi)\mid m\in \mathbb{Z}\}$.
Hence $\tilde{\rho}(x^p\mathcal{L}^q)=(0,0)$.
Then $\tilde{\rho}$ can induce a homomorphism $\pi_1(K(r))\to \widetilde{SL_2(\mathbb{R})}$
with non-abelian image.
Recall that $\widetilde{SL_2(\mathbb{R})}$ is left-orderable (\cite{Be}) and
any (non-trivial) subgroup of a left-orderable group is left-orderable.
Since $K(r)$ is irreducible \cite{HT}, 
$\pi_1(K(r))$ is left-orderable by \cite[Theorem 1.1]{BRW}.
For $r=0$, $K(0)$ is irreducible and has positive betti number.
Hence $\pi_1(K(0))$ is left-orderable by \cite[Corollary 3.4]{BRW}.
Thus we have shown that any slope in $[0,4m)$ is left-orderable for $K=K(m,n)$.
 
If we apply this argument for $K(n,m)$,
then any slope in $[0,4n)$ is shown to be left-orderable.
Since $K(n,m)$ is equivalent to the mirror image of $K(m,n)$,
any slope in $(-4n,0]$ is left-orderable for $K(m,n)$.
Thus $I=(-4n,4m)$ consists of left-orderable slopes for $K=K(m,n)$.

The other cases are proved similarly.
\end{proof}

\bibliographystyle{amsplain}

\begin{thebibliography}{BGW}

\bibitem{B}
V. Bargmann, 
\textit{Irreducible unitary representations of the Lorentz group},
Ann. of Math. \textbf{48} (1947), 568--640. 

\bibitem{Be}
G. Bergman, 
\textit{Right orderable groups that are not locally indicable}, 
Pacific J. Math. \textbf{147} (1991), 243--248. 


\bibitem{BGW}
S. Boyer, C. McA. Gordon and L. Watson,
\textit{On $L$-spaces and left-orderable fundamental groups},
to appear in Math. Ann.

\bibitem{BRW}
S. Boyer, D. Rolfsen and B. Wiest,
\textit{Orderable 3-manifold groups},
Ann. Inst. Fourier (Grenoble) \textbf{55} (2005), 243--288.


\bibitem{BZ}
G. Burde and H. Zieschang,
\textit{Knots}, de Gruyter Studies in Mathematics, 5. Walter de Gruyter \& Co., Berlin, 2003.

\bibitem{CLW}
A. Clay, T. Lidman and L. Watson,
\textit{Graph manifolds, left-orderability and amalgamation},
preprint, \texttt{arXiv:1106.0486}.

\bibitem{CT}
A. Clay and M. Teragaito,
\textit{Left-orderability and exceptional Dehn surgery on two-bridge knots},
to appear in the Proceedings of Geometry and Topology Down Under,
Contemporary Mathematics Series.

\bibitem{DHY}
J. Dubois, Y. Huynh and Y. Yamaguchi,
\textit{Non-abelian Reidemeister torsion for twist knots},
J. Knot Theory Ramifications \textbf{18} (2009), 303--341.

\bibitem{G}
E. Ghys, 
\textit{Groups acting on the circle},
Enseign. Math. \textbf{47} (2001), 329--407. 

\bibitem{HT0}
R. Hakamata and M. Teragaito,
\textit{Left-orderable fundamental group and Dehn surgery on the knot $5_2$},
preprint, \texttt{arXiv:1208.2087}.

\bibitem{HT1}
R. Hakamata and M. Teragaito,
\textit{Left-orderable fundamental group and Dehn surgery on twist knots},
preprint, \texttt{arXiv:1212.6305}.

\bibitem{HT}
A. Hatcher and W. Thurston,
\textit{Incompressible surfaces in 2-bridge knot complements},
Invent. Math. \textbf{79} (1985), 225--246.

\bibitem{HS}
J. Hoste and P. Shanahan,
\textit{A formula for the A-polynomial of twist knots},
J. Knot Theory Ramifications \textbf{13} (2004), 193--209.

\bibitem{Kh}
V. T. Khoi,
\textit{A cut-and-paste method for computing the Seifert volumes},
Math. Ann. \textbf{326} (2003), 759--801. 


\bibitem{MT}
T. Morifuji and A. T. Tran,
\textit{Twisted Alexander polynomilas of $2$-bridge knots
for parabolic representations},
preprint, \texttt{arXiv:1301.1101}.


\bibitem{OS}
P. Ozsv\'{a}th and Z. Szab\'{o},
\textit{On knot Floer homology and lens space surgeries},
Topology \textbf{44} (2005), 1281--1300. 

\bibitem{Ri}
R. Riley,
\textit{Nonabelian representations of 2-bridge knot groups}, 
Quart. J. Math. Oxford Ser. (2) \textbf{35} (1984), 191--208. 



\bibitem{S}
H. Schubert,
\textit{Knoten mit zwei Br\"{u}cken},
Math. Z. \textbf{65} (1956), 133--170. 

\bibitem{T}
M. Teragaito,
\textit{Left-orderability and exceptional Dehn surgery on twist knots},
to appear in Canad. Math. Bull.

\bibitem{Tr}
A. T. Tran,
\textit{On left-orderable fundamental groups and Dehn surgeries on twist knots},
preprint.



\end{thebibliography}

\end{document}